\documentclass[a4paper]{amsart} 
\usepackage{amssymb} 
\usepackage{tikz-cd}
\usepackage{hyperref}

\title{Auslander-Reiten duality for Grothendieck abelian categories}
\author{Henning Krause}
\address{Henning Krause\\ Fakult\"at f\"ur Mathematik\\
  Universit\"at Bielefeld\\ D-33501 Bielefeld\\ Germany.}
\email{hkrause@math.uni-bielefeld.de}

\theoremstyle{plain}
\newtheorem{thm}{Theorem}[section]
\newtheorem{prop}[thm]{Proposition}
\newtheorem{lem}[thm]{Lemma} 

\newtheorem{cor}[thm]{Corollary}

\theoremstyle{definition}
\newtheorem{defn}[thm]{Definition}

\theoremstyle{remark}
\newtheorem{rem}[thm]{Remark}

\numberwithin{equation}{section}

\newcommand{\coind}{\operatorname{coind}}
\newcommand{\Coker}{\operatorname{Coker}}

\newcommand{\eff}{\operatorname{eff}}
\newcommand{\Eff}{\operatorname{Eff}}
\newcommand{\End}{\operatorname{End}}

\newcommand{\Ext}{\operatorname{Ext}}

\newcommand{\fp}{\operatorname{fp}}
\newcommand{\Fp}{\operatorname{Fp}}

\newcommand{\Hom}{\operatorname{Hom}}

\renewcommand{\Im}{\operatorname{Im}}

\newcommand{\Inj}{\operatorname{Inj}}

\newcommand{\Ker}{\operatorname{Ker}}

\newcommand{\Mod}{\operatorname{Mod}}

\newcommand{\oHom}{\operatorname{\overline{Hom}}}

\newcommand{\oMod}{\operatorname{\overline{Mod}}}

\newcommand{\Qcoh}{\operatorname{Qcoh}}

\newcommand{\RHom}{\operatorname{{\mathbf R}Hom}}
\newcommand{\RHOM}{\operatorname{\mathbf{R}\mathcal{H}\!\!\;\mathit{om}}}

\newcommand{\Spec}{\operatorname{Spec}}

\newcommand{\Tr}{\operatorname{Tr}}
\newcommand{\uEnd}{\operatorname{\underline{End}}}
\newcommand{\uHom}{\operatorname{\underline{Hom}}}
\newcommand{\umod}{\operatorname{\underline{mod}}}
\newcommand{\uMod}{\operatorname{\underline{Mod}}}

\newcommand{\Ab}{\mathrm{Ab}}

\newcommand{\op}{\mathrm{op}}

\newcommand{\comp}{\mathop{\circ}}

\newcommand{\Lotimes}{\overset{\mathbf L}{\otimes}}

\newcommand{\lto}{\longrightarrow}

\newcommand{\xto}{\xrightarrow}

\def\A{\mathcal A} 
 
\def\C{\mathcal C}
\def\D{\mathcal D}

\def\bfi{\mathbf i}

\def\bfD{\mathbf D} 
\def\bfK{\mathbf K}

\def\bbX{\mathbb X} 
\def\bbY{\mathbb Y} 
\def\bbZ{\mathbb Z}

\def\a{\alpha}

\def\e{\varepsilon}

\def\g{\gamma}
\def\p{\phi}

\def\s{\sigma}
\def\t{\tau}

\def\Ga{\Gamma}
\def\La{\Lambda}
\def\Si{\Sigma}

\begin{document}

\thanks{April 9, 2016}
\begin{abstract}
  Auslander-Reiten duality for module categories is generalised to
  Grothendieck abelian categories that have a sufficient supply of
  finitely presented objects. It is shown that Auslander-Reiten
  duality amounts to the fact that the functor $\Ext^1(C,-)$ into
  modules over the endomorphism ring of $C$ admits a partially defined
  right adjoint when $C$ is a finitely presented object. This result
  seems to be new even for module categories. For appropriate schemes
  over a field, the connection with Serre duality is discussed.
\end{abstract}


\maketitle


\section{Introduction}

The cornerstones of Auslander-Reiten duality are the following
formulas for modules over an artin algebra $\La$:
\[D\Ext^1_\La(C,-)\cong\oHom_\La(-,D\Tr C)\quad\text{and}\quad
\Ext^1_\La(-,D\Tr C)\cong D\uHom_\La(C,-).\]
They were established by Auslander and Reiten \cite{AR1975} and later
generalised to modules over arbitrary rings \cite{Au1978}. The crucial
ingredient is the explicit construction of the \emph{Auslander-Reiten
  translate} by taking the dual of the transpose $D\Tr C$ of a finitely
presented module $C$.

There are several options to generalise this. A very elegant approach
due to Bondal and Kapranov \cite{BK1989} uses the notion of a \emph{Serre
functor} for a triangulated category. In particular, this reveals the
close connection between Auslander-Reiten duality and Serre duality.

For abelian categories, Auslander and Reiten established in
\cite{AR1974} a generalisation by introducing the concept of a
\emph{dualising variety}. Further approaches include the work of
Reiten and Van den Bergh for hereditary abelian categories
\cite{RV2002} and that of Lenzing and Zuazua \cite{LZ2004}.
 
The aim of this paper is to establish Auslander-Reiten duality more
generally for Grothendieck abelian categories that have a sufficient
supply of finitely presented objects. We mimick the construction of
the Auslander-Reiten translate by invoking the existence of flat
covers in certain functor categories. In fact, we show that the
Auslander-Reiten translate is the representing object of a specific
functor; so our approach is somewhat similar to Neeman's acccount of
Grothendieck duality via Brown representability \cite{Ne1996}.

Motivated by the general setting of Grothendieck abelian categories,
we obtain a coherent formulation of Auslander-Reiten duality which
seems to be new even for module categories.  For a finitely presented
module we provide an explicit construction of the Auslander-Reiten
translate; it is a modification of the original construction due to
Auslander and Reiten.

For an abelian category $\A$ let $\overline\A$ denote the \emph{stable
  category modulo injectives}, which is obtained from $\A$ by
identifying two morphisms $\p,\p'\colon X\to Y$ if
\[\Ext^1_\A(-,\p)=\Ext^1_\A(-,\p').\] 
When $\A$ has enough injective objects this means $\p-\p'$ factors
through an injective object. We write $\oHom_\A(-,-)$ for the
morphisms in $\overline\A$. Analogously, the \emph{stable category
  modulo projectives} $\underline\A$ is defined.

\begin{thm}\label{th:intro}
  Let $\A$ be a Grothendieck abelian category that is locally finitely
  presented. Fix a finitely presented object $C$ and set
  $\Ga=\End_\A(C)$. Then for every injective $\Ga$-module $I$ there
  exists an object $\t_C(I)$ in $\overline\A$ and a natural
  isomorphism
\begin{equation}\label{eq:repres}
\Hom_\Ga(\Ext^1_\A(C,-),I)\cong\oHom_\A(-,\t_C(I)).
\end{equation}
\end{thm}

We refer to Theorem~\ref{th:main} for the proof of this result and
continue with some consequences.

Another version of Auslander-Reiten duality involves the stable
category modulo projectives. We say that $\A$ has \emph{enough
  projective morphisms} if every object $X$ in $\A$ admits an
epimorphism $\pi\colon X' \to X$ such that $\Ext^1_\A(\pi,-)=0$.

\begin{cor}\label{co:intro1}
  For an injective $\Ga$-module $I$ there is a natural
  monomorphism 
\begin{equation}\label{eq:proj-repres}
\Ext^1_\A(-,\t_C(I))\lto  \Hom_\Ga(\uHom_\A(C,-),I)
\end{equation}
which is an isomorphism when $\A$ has enough projective morphisms.
\end{cor}

A necessary and sufficient condition for \eqref{eq:proj-repres} to be
an isomorphism is given in Theorem~\ref{th:second-main}.  For an
intriguing symmetry between \eqref{eq:repres} and
\eqref{eq:proj-repres}, see Appendix~\ref{ap:ext}. Note that the
r\^{o}les of \eqref{eq:repres} and \eqref{eq:proj-repres} are quite
different. The first one provides the Auslander-Reiten translate as a
representing object, while the second seems to be more suitable for
applications. For instance, \eqref{eq:proj-repres} is used for
constructing almost split sequences. Also, \eqref{eq:proj-repres}
identifies with Serre duality for categories of quasi-coherent sheaves
over projective schemes.

There is an equivalent formulation of the isomorphism
\eqref{eq:repres} in terms of the
defect of an exact sequence \cite{Au1978}.  Given an exact
sequence \[\xi\colon 0\lto X\lto Y\lto Z\lto 0\]
in $\A$, the \emph{covariant defect} $\xi_*$ and the
\emph{contravariant defect} $\xi^*$ are defined by the exactness of
the following sequences:
\begin{gather*}
0\lto \Hom_\A(Z,-)\lto\Hom_\A(Y,-)\lto\Hom_\A(X,-)\lto \xi_*\lto 0\\
0\lto \Hom_\A(-,X)\lto\Hom_\A(-,Y)\lto\Hom_\A(-,Z)\lto \xi^*\lto 0
\end{gather*}

\begin{cor}
  For an injective $\Ga$-module $I$ there is a natural
  isomorphism \[\Hom_\Ga(\xi^*(C),I)\cong\xi_*(\t_C(I)).\]
\end{cor}

In applications one often deals with an abelian category that is
$k$-linear over a commutative ring $k$. In that case the isomorphism
\eqref{eq:repres} gives for any injective $k$-module $I$ a natural
isomorphism
\begin{equation}\label{eq:k-repres}
\Hom_k(\Ext^1_\A(C,-),I)\cong\oHom_\A(-,\t(C,I))
\end{equation}
by setting $\t(C,I)=\t_C(\Hom_k(\Ga,I))$. For instance, when $\A$ is
the category of modules over a $k$-algebra $\La$ and $C$ is a finitely
presented $\La$-module, then 
\[\t(C,I)=\Hom_k(\Tr C,I).\]

A case of particular interest is given by a non-singular projective
scheme $\bbX$ of dimension $d\ge 1$ over a field $k$. For the category
$\A$ of quasi-coherent $\mathcal O_\bbX$-modules and a coherent
$\mathcal O_\bbX$-module $C$ we
have \[\tau(C,k)=\Si^{d-1}(C\otimes_{\bbX}\omega_\bbX)\]
where $\omega_\bbX$ is the dualising sheaf and $\Si^{d-1}$ denotes the
$(d-1)$st syzygy in a minimal injective resolution. In that case the
isomorphism \eqref{eq:k-repres} is a variation of Serre duality
\cite{Gr1968,Se1955}. In fact, a more familiar form of Serre duality
is given by the natural isomorphism
\[\Ext^{1}_{\bbX}(-, \Si^{d-1}(C\otimes_\bbX\omega_\bbX))\cong
\Ext^{d}_{\bbX}(-, C\otimes_\bbX\omega_\bbX)\cong
\Hom_k(\Hom_\bbX(C,-),k)\]
which identifies with \eqref{eq:proj-repres} since
$\uHom_\bbX(C,-)=\Hom_\bbX(C,-)$. This provides a precise connection
between Auslander-Reiten duality and Serre duality.

For an arbitrary Grothendieck abelian category the construction of the
\emph{Auslander-Reiten translate} $\t_C(I)$ is far from explicit; it
involves the existence of flat covers which was an open problem for
about twenty years \cite{BEE2001}.\footnote{There is an alternative
  proof of Theorem~\ref{th:intro} which obtains the Auslander-Reiten
  translate from Brown representability, using the fact that the
  homotopy category of complexes of injective objects $\bfK(\Inj\A)$
  is a compactly generated triangulated category \cite{Kr2015,St2014}.}  When
$\A$ is the category of modules over a ring, we provide an explicit
description of the Auslander-Reiten translate, making also the
connection with the dual of the transpose of Auslander and Reiten; see
Definition~\ref{de:DTr} and Theorem~\ref{th:defect1}.

This paper has three parts. First we deal with the general case of a
Grothendieck abelian category, then we consider the Auslander-Reiten
translate for module categories, and the final section is devoted to
Auslander-Reiten duality for categories of quasi-coherent sheaves.

\section{Auslander-Reiten duality for Grothendieck abelian categories}

In this section we introduce the Auslander-Reiten translate for
Grothendieck abelian categories that are locally finitely presented
(Theorem~\ref{th:main}).

Following \cite{Br1970}, a Grothendieck abelian category $\A$ is
\emph{locally finitely presented} if the finitely presented objects
generate $\A$.  Recall that an object $X$ in $\A$ is \emph{finitely
  presented} if the functor $\Hom_\A(X,-)$ preserves filtered
colimits.  We denote by $\fp\A$ the full subcategory of finitely
presented objects in $\A$.  Note that the isomorphism classes of
finitely presented objects form a set when $\A$ is locally finitely
presented.

We begin with some preparations.

\subsection*{Modules}

Let $\A$ be an additive category. We write $(\A^\op,\Ab)$ for the
category of additive functors $\A^\op\to\Ab$ where $\Ab$ denotes the
category of abelian groups. The morphisms between functors are the
natural transformations and we obtain an abelian category.  Note that
(co)kernels and (co)products are computed pointwise: for instance, a
sequence $X\to Y\to Z$ of morphisms in $(\A^\op,\Ab)$ is exact if and
only if the sequence $X(A)\to Y(A)\to Z(A)$ is exact in $\Ab$ for all
$A$ in $\A$. When $\A$ is essentially small, then the morphisms
between two functors in $(\A^\op,\Ab)$ form a set.

Now fix a set $\C$ of objects in $\A$ and view $\C$ as a full
subcategory of $\A$. We set $\Mod\C=(\C^\op,\Ab)$ and call the objects
\emph{$\C$-modules}. For example, if $\C$ consists of one object $C$,
then $\Mod\C$ is the category of modules over the endomorphism ring of
$C$.

\subsection*{Restriction and coinduction}

Let $\A$ be an additive category. For a full subcategory
$\C\subseteq\A$ there is the \emph{restriction functor}
\[(\A^\op,\Ab)\lto (\C^\op,\Ab),\qquad F\mapsto F|_\C\]
and its right adjoint, the \emph{coinduction functor}
\[\coind_\C\colon (\C^\op,\Ab)\lto (\A^\op,\Ab)\]
given by
\[\coind_\C I(X)=\Hom(\Hom_\A(-,X)|_\C,I)\qquad \text{for }I\in
(\C^\op,\Ab),\,X\in\A.\]
Thus for $F\in (\A^\op,\Ab)$ and $I\in (\C^\op,\Ab)$, there is an isomorphism
\begin{equation}\label{eq:coind}
\Hom(F|_\C,I)\cong \Hom(F,\coind_\C I).
\end{equation}

\begin{lem}\label{le:coind}
  The functor $\coind_\C$ preserves injectivity.
\end{lem}
\begin{proof}
The restriction functor is exact, and any right adjoint of an exact
functor preserves injectivity.
\end{proof}

\subsection*{Finitely presented functors}

Let $\A$ be an additive category. We denote by $\Fp(\A^\op,\Ab)$ the
category of finitely presented functors $F\colon\A^\op\to\Ab$. Recall
that $F$ is \emph{finitely presented} (or \emph{coherent}) if it fits
into an exact sequence
\[\Hom_\A(-,X)\lto \Hom_\A(-,Y)\lto F\lto 0.\]
Note that $\Fp(\A^\op,\Ab)$ is an abelian category when $\A$ admits
kernels. Then the assignment $X\mapsto\Hom_\A(-,X)$ identifies $\A$ with
the full subcategory of projective objects in $\Fp(\A^\op,\Ab)$ by
Yoneda's lemma.

\subsection*{Flat functors and flat covers}

Let $\C$ be an essentially small additive category. We consider the
category $(\C^\op,\Ab)$ of additive functors
$F\colon\C^\op\to\Ab$. Recall that $F$ is \emph{flat} if it can be
written as a filtered colimit of representable functors.

The following result describes the connection between locally finitely
presented categories and categories of flat functors.

\begin{thm}[{Breitsprecher \cite{Br1970}}]\label{th:flat}
  Let $\A$ be a locally finitely presented Grothendieck abelian
  category. Then the functor
  \begin{equation*}\label{eq:flat}
\A\lto ((\fp\A)^\op,\Ab),\qquad X\mapsto
  \Hom_\A(-,X)|_{\fp\A}
\end{equation*} 
identifies $\A$ with the full subcategory of flat functors
$(\fp\A)^\op\to\Ab$. Moreover, the functor admits an exact left
adjoint.\qed
\end{thm}

A morphism $\pi\colon F\to G$ in $(\C^\op,\Ab)$ is a \emph{flat
  cover} of $G$ if the following holds:
\begin{enumerate}
\item $F$ is flat and every morphism $F'\to G$ with $F'$ flat factors
  through $\pi$.
\item $\pi$ is \emph{right minimal}, that is, an endomorphism
  $\p\colon F\to F$ satisfying $\pi\p=\pi$ is invertible.
\end{enumerate}
A \emph{minimal flat presentation} of $G$ is an exact sequence
\[F_1\lto F_0\stackrel{\pi}\lto G\lto 0\] such that $F_0\to G$ and
$F_1\to\Ker\pi$ are flat covers.  A \emph{projective cover} and a
\emph{minimal projective presentation} are defined analogously,
replacing the term flat by projective.\footnote{This definition of a
  projective cover is equivalent to the usual one
  which requires the kernel to be superfluous.}

\begin{thm}[{Bican--El Bashir--Enochs \cite{BEE2001}}]
Every additive functor $\C^\op\to\Ab$ admits a flat cover.\qed
\end{thm}

The following consequence is straightforward; see
\cite[Theorem~2.2]{Kr2014}.

\begin{cor}\label{co:flatcover}
Let $\A$ be a locally finitely presented Grothendieck abelian
category. Then every functor $F\colon \A^\op\to\Ab$ that preserves
filtered colimits belongs to $\Fp(\A^\op,\Ab)$
and admits a minimal projective presentation.
\end{cor}
\begin{proof}
Choose a minimal flat presentation of $F|_{\fp\A}$ and apply Theorem~\ref{th:flat}.
\end{proof}

The next lemma will be needed to identify injective objects in
a locally finitely presented Grothendieck abelian category.

\begin{lem}\label{le:domdim}
Let $\C$ be an essentially small additive category and consider the
following conditions in  $(\C^\op,\Ab)$.
\begin{enumerate}
\item Given a minimal injective copresentation $0\to G\to I^0\to I^1$
  such that
$G$ is flat, then $I^0$ and $I^1$ are flat.
\item Given a minimal flat presentation $F_1\to F_0\to G\to 0$ such
  that $G$ is injective, then $F_0$ and $F_1$ are injective.
\end{enumerate}
Then \emph{(1)} implies \emph{(2)}.
\end{lem}
\begin{proof}
  Fix a minimal flat presentation $F_1\to F_0\xto{\pi} G\to 0$ such
  that $G$ is injective.  Let $F_0\to E(F_0)$ be an injective
  envelope. Then $\pi$ factors through this since $G$ is injective. On
  the other hand, the morphism $E(F_0)\to G$ factors through $\pi$
  since $E(F_0)$ is flat. The minimality of $\pi$ implies that $F_0$
  is a direct summand of $E(F_0)$, and therefore $F_0$ is
  injective. Now choose a minimal injective copresentation
  $0\to F_1\to I^0\to I^1$. This gives rise to the following
  commutative diagram with exact rows.
\[\begin{tikzcd}
F_1\arrow{r}\arrow{d}& F_0\arrow{r}\arrow{d}&G\arrow{r}\arrow{d}&0\\
I^0\arrow{r}& F_0\oplus I^1\arrow{r}&H\arrow{r}&0
\end{tikzcd}\]
The vertical morphisms are monomorphism, and therefore $G\to H$
splits. The inverse $H\to G$ induces the following 
commutative diagram with exact rows since $I^0$ and $I^1$ are flat.
\[\begin{tikzcd}
I^0\arrow{r}\arrow{d}&  F_0\oplus I^1\arrow{r}\arrow{d}&H\arrow{r}\arrow{d}&0\\
F_1\arrow{r}& F_0\arrow{r}&G\arrow{r}&0
\end{tikzcd}\]
Now the minimality of the flat presentation implies that the
composition $F_1\to I^0\to F_1$ is invertible. Thus $F_1$ is injective.
\end{proof}

Let $\A$ be a locally finitely presented Grothendieck abelian
category. Set $\C=\fp\A$ and consider the functor
$h\colon\A\to (\C^\op,\Ab)$ from Theorem~\ref{th:flat}.  Then $X\in\A$
is injective if and only if $h(X)$ is injective, since $h$ has an
exact left adjoint.  In fact, $h$ takes an injective copresentation
in $\A$ to one in $(\C^\op,\Ab)$. Thus $\C$ satisfies condition (1) in
Lemma~\ref{le:domdim}.

\begin{cor}\label{co:domdim}
  Let $\A$ be a locally finitely presented Grothendieck abelian
  category and consider a minimal projective presentation
\[\Hom_\A(-,X_1)\lto \Hom_\A(-,X_0)\lto F\lto 0\] of a functor $F$ in
$\Fp(\A^\op,\Ab)$. If $F|_{\fp\A}$ is an injective object in
$((\fp\A)^\op,\Ab)$, then $X_0$ and $X_1$ are injective objects in $\A$.
\end{cor}
\begin{proof}
The sequence
\[\Hom_\A(-,X_1)|_{\fp\A}\lto \Hom_\A(-,X_0)|_{\fp\A}\lto F|_{\fp\A}\lto 0\]
is a minimal flat presentation of $F|_{\fp\A}$ in
$((\fp\A)^\op,\Ab)$. Thus the assertion follows from
Lemma~\ref{le:domdim}.
\end{proof}

\subsection*{The defect of an exact sequence}

We recall the following notion from \cite[II.3]{Au1978}.  Let $\A$ be
an abelian category.  For an exact sequence
\[\xi\colon 0\lto X\lto Y\lto Z\lto 0\]
in $\A$ the \emph{covariant defect} $\xi_*$ and the
\emph{contravariant defect} $\xi^*$ are defined by the exactness of
the following sequences:
\begin{gather*}
0\lto \Hom_\A(Z,-)\lto\Hom_\A(Y,-)\lto\Hom_\A(X,-)\lto \xi_*\lto 0\\
0\lto \Hom_\A(-,X)\lto\Hom_\A(-,Y)\lto\Hom_\A(-,Z)\lto \xi^*\lto 0
\end{gather*}

The functors $\xi_*\colon\A\to\Ab$ given by the exact sequences $\xi$
in $\A$ form an abelian category, with morphisms the natural
transformations. We denote this category by $\Eff(\A,\Ab)$, because
the objects are precisely the finitely presented functors
$F\colon\A\to\Ab$ that are \emph{locally effaceable} \cite{Gr1957},
that is, for each $x\in F(X)$ there exists a monomorphism
$\p\colon X\to Y$ such that $F(\p)(x)=0$. The assignment
$F\mapsto D_\A(F)$ given by
\[D_\A(F)(X)=\Ext^2(F,\Hom_\A(X,-))\] yields an equivalence
\begin{equation}\label{eq:eff}
D_\A\colon \Eff(\A,\Ab)^\op\xto{\ \sim\ }\Eff(\A^\op,\Ab),
\end{equation}
where $\Ext^2(-,-)$ is computed in the abelian category $\Fp(\A,\Ab)$
and the inverse is given by $D_{\A^\op}$. Note that
$D_\A(\xi_*)=\xi^*$ and $D_{\A^\op}(\xi^*)=\xi_*$. When the context is
clear we write $D$ instead of $D_\A$.

We continue with some further properties of locally effaceable
functors. For a discussion of the following result, see also \cite{Oo1963}.

\begin{lem}\label{le:hom-ext}
  Suppose that every object $X\in\A$ admits a monomorphism
  $\iota \colon X\to Y$ such that $\Ext^1_\A(-,\iota)=0$. Then the
  assignment $\p\mapsto\Ext^1_\A(-,\p)$ induces for all objects $X,X'\in\A$ an isomorphism
\[\oHom_\A(X,X')\xto{\ \sim\ }\Hom(\Ext^1_\A(-,X),\Ext^1_\A(-,X')).\]
\end{lem}

Clearly, the assumption on $\A$ is satisfied when $\A$ has enough
injective objects.
\begin{proof}
Let $\xi\colon 0\to X\xto{\iota} Y\to Z\to 0$ be an exact sequence
such that $\Ext^1_\A(-,\iota)=0$. Then we have
\[\xi^*\cong\Ext^1_\A(-,X)\qquad\text{and}\qquad\xi_*\cong\oHom_\A(X,-).\]
Now apply the duality \eqref{eq:eff}.
\end{proof}

\begin{lem}\label{le:eff-adj}
The inclusion $\Eff(\A^\op,\Ab)\to \Fp(\A^\op,\Ab)$ admits a right adjoint
\begin{equation*}
\eff\colon \Fp(\A^\op,\Ab)\lto \Eff(\A^\op,\Ab).
\end{equation*}
\end{lem}
\begin{proof}
The right adjoint  takes a functor $F=\Coker \Hom_\A(-,\p)$ given by a morphism
$\p\colon X\to Y$ in $\A$ to $\Coker \Hom_\A(-,\p')$ where $\p'\colon
X\to\Im\p$ is the morphism induced by $\p$.
\end{proof}

\subsection*{Auslander-Reiten duality}

Let $\A$ be an abelian category and $\C$ a set of objects in
$\A$. Then every object $X$ in $\A$ gives rise to a $\C$-module
\[\Ext^1_\A(\C, X) := \Ext^1_\A(-, X)|_\C.\]

\begin{thm}\label{th:main}
  Let $\A$ be a locally finitely presented Grothendieck abelian
  category and $\C$ a set of finitely presented objects in $\A$.  Then
  for every injective $\C$-module $I$ there exists an object
  $\t_\C(I)$ in $\overline\A$ and a natural isomorphism
\begin{equation*}\label{eq:tau}
  \Hom(\Ext^1_\A(\C,-),I)\cong\oHom_{\A}(-,\tau_\C(I)).
\end{equation*}
\end{thm}

The special case that $\C$ consists of a single object is precisely
Theorem~\ref{th:intro} from the introduction.  The theorem suggests
the following definition.

\begin{defn}\label{de:tau}
  The object $\tau_\C(I)$ is called the \emph{Auslander-Reiten
    translate} with respect to $\C\subseteq\fp\A$ and $I\in\Mod\C$.
  The assignment $I\mapsto\t_\C(I)$ yields a functor
\[\t_\C\colon\Inj\Mod\C\lto\overline\A.\]
\end{defn}

\begin{proof}[Proof of Theorem~\ref{th:main}]
  We fix $\C\subseteq\fp\A$ and $I\in\Mod\C$. The functor
  $\coind_\C(I)$ in $(\A^\op,\Ab)$ preserves filtered colimits and
  admits therefore a minimal presentation
  \[0\to\Hom_\A(-,T_2)\to\Hom_\A(-,T_1)\to\Hom_\A(-,T_0)\to
  \coind_\C(I)\to 0\]
  by Corollary~\ref{co:flatcover}.  In $\A$ this gives rise to an exact
  sequence
\[\eta\colon 0\lto T_2 \lto T_1\lto \bar T_0\lto 0.\] 
We set
\[\t_\C(I):=T_2.\]
Note that $T_1$ is an injective object when $I$ is injective. This
follows from Corollary~\ref{co:domdim}, because $\coind_\C(I)$ is an
injective object in $((\fp\A)^\op,\Ab)$ by Lemma~\ref{le:coind}. Thus
\begin{equation}\label{eq:eta}
\eff\coind_\C(I)\cong \eta^*\cong\Ext^1_\A(-,\t_\C(I)).
\end{equation}

Now  fix an object $X\in\A$. Then we have
\begin{align*} 
\Hom(\Ext^1_\A(-,X)|_\C,I)
&\cong\Hom(\Ext^1_\A(-,X),\coind_\C(I))\\
&\cong\Hom(\Ext^1_\A(-,X),\eff\coind_\C(I))\\
&\cong\Hom(\Ext^1_\A(-,X),\Ext^1_\A(-,\t_\C(I)))\\
&\cong\oHom_{\A}(X,\tau_\C(I)).
\end{align*}
Let us label the $n$th isomorphism by ($n$).  Then (1) and (2) are
adjunctions, given by \eqref{eq:coind} and Lemma~\ref{le:eff-adj}, (3)
uses \eqref{eq:eta}, and (4) follows from Lemma~\ref{le:hom-ext}.  This
completes the proof.
\end{proof}

One may wonder whether the functor
$\Ext_\A^1(\C,-)\colon\overline\A\to\Mod\C$ admits a right adjoint. In fact,
the proof of Theorem~\ref{th:main} provides for an arbitrary
$\C$-module $I$ an object $\t_\C(I)$ and a natural monomorphism
 \[\Hom(\Ext^1_\A(\C,-),I)\lto\oHom_{\A}(-,\tau_\C(I)).\] 
It is easily seen that this is not invertible in general.

\begin{rem}
  Let $\C\subseteq\D\subseteq\fp\A$ and denote by $i\colon\C\to\D$ the
  inclusion. Then we have $\tau_\C=\tau_\D\comp i_*$ where $i_*$ is
  the right adjoint of restriction $\Mod\D\to\Mod\C$.
\end{rem}

\begin{rem}\label{re:stable}
  Consider the stable category modulo projectives $\underline\A$. For
  a set of objects $\C\subseteq\A$ let $\underline\C$ denote the
  corresponding subcategory of $\underline\A$. Then we have
\[\Ext^1_\A(\C,X)\in\Mod\underline\C\subseteq\Mod\C\]
for all $X\in\A$ and one can replace $\Mod\C$ by $\Mod\underline\C$ in
Theorem~\ref{th:main}.
\end{rem}

\subsection*{Auslander-Reiten duality relative to a base}

Let $k$ be a commutative ring. Suppose that $\A$ is a $k$-linear
and locally finitely presented Grothendieck abelian category.

\begin{cor}\label{co:base}
There is a functor
\[\fp\A\times\Inj\Mod k\lto\overline\A,\qquad (C,I)\mapsto\tau(C,I)\] and
a natural isomorphism
\[
 \Hom_{k}(\Ext^1_\A(C,-),I)\cong\oHom_{\A}(-,\tau(C,I)).
\]
\end{cor}
\begin{proof}
Fix $C\in\fp\A$ and set $\Ga=\End_\A(C)$. Observe that $\Hom_k(\Ga,-)$
induces a functor $\Mod k\to\Mod\Ga$ that takes injectives to injectives.
Then we obtain $\tau(C,-)$ from the Auslander-Reiten translate $\tau_C$
 by setting
\[\tau(C,-)=\tau_C(\Hom_k(\Ga,-)).\qedhere\]
\end{proof}

\subsection*{A formula for the defect}

There is a reformulation of the isomorphism in Theorem~\ref{th:main}
in terms of the defect of an exact sequence.

\begin{thm}\label{th:defect}
  Let $\A$ be a locally finitely presented Grothendieck abelian
  category. Fix a set $\C$ of finitely presented objects and an exact
  sequence $\xi\colon 0\to X\to Y\to Z\to 0$ in $\A$. Then for every
  injective $\C$-module $I$ there is a natural isomorphism
\[\Hom(\xi^*|_\C,I)\cong\xi_*(\tau_\C(I)).\]
\end{thm}

The above isomorphism can be rewritten using the duality
\eqref{eq:eff}. Thus we obtain for $F$ in $\Eff(\A^\op,\Ab)$ a natural
isomorphism
\[
\Hom(F|_\C,I)\cong D(F)(\tau_\C(I)).
\]

\begin{proof}[Proof of Theorem~\ref{th:defect}]
  Fix $\xi\colon 0\to X\xto{\p} Y\xto{\psi} Z\to 0$.
  It is not hard to see that
\[\xi_*\cong\Coker \oHom_\A(\p,-)\qquad\text{and}\qquad\xi^*\cong
\Ker\Ext_\A^1(-,\p).\]
Combining this with the isomorphism in Theorem~\ref{th:main} we get
\begin{align*}
\Hom(\xi^*|_\C,I) &\cong\Hom(\Ker\Ext_\A^1(\C,\p),I)\\
&\cong\Coker\Hom(\Ext_\A^1(\C,\p),I)\\
&\cong\Coker\oHom_\A(\p,\tau_\C(I))\\
&\cong \xi_*(\tau_\C(I)).\qedhere
\end{align*}
\end{proof}

\subsection*{Auslander-Reiten duality  modulo
  projectives}

Let $\A$ be an abelian category. Recall that the \emph{stable category
  modulo projectives} $\underline\A$ is obtained  from $\A$ by
identifying two morphisms $\p,\p'\colon X\to Y$ if 
\[\Ext^1_\A(\p,-)=\Ext^1_\A(\p',-).\]
We write $\uHom_\A(-,-)$ for the morphisms in $\underline\A$.

The following result is a refinement of Corollary~\ref{co:intro1} from
the introduction.

\begin{thm}\label{th:second-main}
  Let $\A$ be a locally finitely presented Grothendieck abelian
  category and $\C$ a set of finitely presented objects in $\A$.  Then
  for every injective $\C$-module $I$ there is a natural monomorphism
\begin{equation*}
\a\colon \Ext^1_\A(-,\t_\C(I))\lto  \Hom(\uHom_\A(\C,-),I).
\end{equation*}
This is an isomorphism for all $I$ if and only if for every object
$X\in\A$ there exists an epimorphism $\pi\colon X_\C\to X$ such that for
every morphism $\p\colon C\to X$ with $C\in\C$
\[\p\text{ factors through }\pi\qquad\iff\qquad\Ext^1_\A(\p,-)=0.\]
\end{thm}

The condition for $\a$ to be an isomorphism expresses the fact that
$\A$ has locally enough projective morphisms. Clearly, the condition
is satisfied when $\A$ has enough projective objects, but there are
also other examples; see Proposition~\ref{pr:proj-scheme}.

\begin{proof}[Proof of Theorem~\ref{th:second-main}]
For $X\in\A$ we have 
\begin{align*}
\Ext^1_\A(X,\t_\C(I))&\cong\Hom(\uHom_\A(-,X),\Ext^1_\A(-,\t_\C(I))\\
&\cong\Hom(\uHom_\A(-,X),\eff\coind_\C(I))\\
&\subseteq\Hom(\uHom_\A(-,X),\coind_\C(I))\\
&\cong\Hom(\uHom_\A(-,X)|_\C,I).
\end{align*}
Let us label the $n$th isomorphism by ($n$). Then (1) uses Yoneda's
lemma, (2) follows from \eqref{eq:eta}, and (3) is the adjunction
\eqref{eq:coind}.  

Now assume that there exists an epimorphism $\pi\colon X_\C\to X$ such
that for \[F=\Coker\Hom_\A(-,\pi)\] we have
$F|_\C\cong\uHom_\A(-,X)|_\C$. Then $F$ is effaceable and we can apply
Lemma~\ref{le:eff-adj}. Thus
\begin{align*}
\Hom(\uHom_\A(-,X)|_\C,I)&\cong \Hom(F|_\C,I)\\
&\cong \Hom(F,\coind_\C(I))\\
&\cong \Hom(F,\eff\coind_\C(I))\\
&\cong \Hom(F,\Ext^1_\A(-,\t_\C(I))\\
&\subseteq\Hom(\Hom_\A(-,X),\Ext^1_\A(-,\t_\C(I))\\
&\cong \Ext^1_\A(X,\t_\C(I))
\end{align*}
and we conclude that $\a$ is an isomorphism.

Finally, assume that $\a$ is an isomorphism for all $I$. An injective
envelope $\eta\colon\uHom_\A(\C,X)\to J$ corresponds under $\a$ to an extension
\[0\lto\t_\C(J)\lto X_\C\xto{\ \pi\ }X\lto 0.\]
From the functoriality of $\a$ one sees that any morphism
$\p\colon C\to X$ factors through $\pi$ iff the the composition of
$\uHom_\A(\C,\p)$ with $\eta$ equals zero. The latter condition means
$\Ext^1_\A(\p,-)=0$.
\end{proof}

\begin{rem}\label{re:pi}
For $X\in\A$ there exists an essentially unique morphism
$\pi\colon X_\C\to X$ having the following properties:
\begin{enumerate}
\item A morphism $\p\colon C\to X$ with $C\in\C$
factors through $\pi$ iff $\Ext^1_\A(\p,-)=0$.
\item  Every morphism $X'\to X$ satisfying (1) factors through $\pi$.
\item Every endomorphism
  $\e\colon X_\C\to X_\C$ satisfying $\pi\e=\pi$ is invertible.
\end{enumerate}
This follows from \cite[Theorem~1.1]{Kr2014}. Thus the crucial issue
for Theorem~\ref{th:second-main} is the property of $\pi$ to be an
epimorphism.
\end{rem}

\section{Auslander-Reiten duality for module categories}

This section is devoted to giving an explicit construction of the Auslander-Reiten translate
for module categories (Definition~\ref{de:DTr}). This is closely
related to the original construction of Auslander and Reiten
\cite{Au1978,AR1974} but it is not the same.

\subsection*{Stable module categories and the transpose}

Let $\La$ be a ring.  Given $\La$-modules $X$ and $Y$, we set
\[\uHom_\La(X,Y)=\Hom_\La(X,Y)/\{\p\mid \p\text{ factors through a
  projective module}\}\]
and
\[\oHom_\La(X,Y)=\Hom_\La(X,Y)/\{\p\mid \p\text{ factors through an
  injective module}\}.\]
Changing not the objects but the morphisms in $\Mod\La$, we obtain the
\emph{stable category modulo projectives} $\uMod\La$. Let
$\umod\La$ denote the full subcategory of finitely presented
$\La$-modules.  Analogously, the \emph{stable category modulo
  injectives} $\oMod\La$ is defined.

For a finitely presented $\La$-module $X$ having a projective presentation
\[P_1\lto P_0\lto X\lto 0\]
the \emph{transpose} $\Tr X$ is defined by the exactness of the
following sequence of $\La^\op$-modules
\[P_0^*\lto P_1^*\lto \Tr X\lto 0\]
where $P^*=\Hom_\La(P,\La)$.

\begin{lem}\label{le:Tr}
The transpose induces mutually inverse equivalences
\[(\umod\La)^\op\xto{\ \sim\ }\umod(\La^\op)\qquad\text{and}\qquad
\umod(\La^\op)\xto{\ \sim\ }(\umod\La)^\op.\]
\end{lem}
\begin{proof}
 See Proposition~2.6 in \cite{AB1969}.
\end{proof}

\subsection*{Injective modules over quotient rings}

Let $\pi\colon \Ga\to\bar\Ga$ be a surjective ring homomorphism. Then
restriction of scalars along $\pi$ yields a functor
$\Mod\bar\Ga\to\Mod \Ga$ that is fully faithful.  This functor has a
right adjoint which takes a $\Ga$-module $I$ to
$\bar I=\Hom_\Ga(\bar\Ga,I)$. Thus $\pi$ induces a monomorphism
$\e_I\colon\bar I\to I$ in $\Mod\Ga$ which identifies $\bar I$ with
the largest submodule of $I$ that is annihilated by $\Ker\pi$ and
gives the isomorphism
\begin{equation}\label{eq:DTr-inj}
\Hom_{\bar\Ga}(-,\bar I)\xto{\sim} \Hom_\Ga(-,I).
\end{equation}

Let $\bar I\to E(\bar I)$ denote an injective envelope in $\Mod\Ga$. 

\begin{lem}\label{le:DTr-inj}
For  an injective $\Ga$-module $I$, we have in $\Mod\bar\Ga$
\[\bar I\xleftarrow{\sim} \Hom_\Ga(\bar\Ga,\bar I) \xto{\sim}\Hom_\Ga(\bar\Ga,E(\bar I)).\]
\end{lem}
\begin{proof}
  The first isomorphism is given by $\e_{\bar I}$. The functor
  $\Hom_\Ga(\bar\Ga,-)$ preserves injectivity since it is right
  adjoint to an exact functor. Thus
  $\bar I\to\Hom_\Ga(\bar\Ga,E(\bar I))$ is a split monomorphism, and
  it is an isomorphism, because the composition with $\e_{E(\bar I)}$
  is an injective envelope.
\end{proof}

\subsection*{The dual of the transpose}

Let $\La$ be a ring and fix a finitely presented $\La$-module $C$. Set
$\Ga=\End_\La(C)$ and $\bar\Ga=\uEnd_\La(C)$.  The transpose $\Tr C$
is a $\La^\op$-module and there is an isomorphism
\[\g\colon\uEnd_\La(C)\xto{\ \sim\ }\uEnd_{\La^\op}(\Tr C)^\op\] by Lemma~\ref{le:Tr}.
For an injective $\Ga$-module $I$ we view $\bar I=\Hom_\Ga(\bar\Ga,I)$
as an $\End_{\La^\op}(\Tr C)^\op$-module via $\g$ and denote by
$E(\bar I)$ an injective envelope. This yields the assignment
\[\Inj \Mod\End_\La(C)\lto \Inj \Mod\End_{\La^\op}(\Tr C)^\op,\qquad
I\mapsto E(\bar I).\]

\begin{defn}\label{de:DTr}
  The \emph{dual of the transpose} (or
  \emph{Auslander-Reiten translate}) of $C$ with respect to $I$ is
  the $\La$-module
\[\t_C(I):=\Hom_{\End_{\La^\op}(\Tr C)^\op}(\Tr C,E(\bar I)).\]
\end{defn}

In Corollary~\ref{co:ARformula} we will see that this definition is
consistent with the previous Definition~\ref{de:tau}. In particular,
the definition does not depend on any choice when $\t_C(I)$ is viewed
as an object in $\oMod\La$.

The definition of $\t_C(I)$ is a variation of Auslander's definition
of the dual of the transpose in \cite[I.3]{Au1978}. To be precise,
Auslander starts with an injective module over
$\End_{\La^\op}(\Tr C)^\op$ whereas the above definition takes as
input an injective module over $\End_\La(C)$. Keeping this difference
in mind, the following is the analogue of Theorem~III.4.1 in
\cite{Au1978}.

\begin{thm}\label{th:defect1}
  Let $\xi\colon 0\to X\to Y\to Z\to 0$ be an exact sequence of
  $\La$-modules and $C$ a finitely presented $\La$-module. 
For an injective $\End_\La(C)$-module $I$,  there
  is a natural isomorphism
\[\Hom_{\End_\La(C)}(\xi^*(C),I)\cong \xi_*(\t_C(I)).\]
\end{thm}
\begin{proof}
  Set $\Si=\End_{\La^\op}(\Tr C)^\op$ and
  $\bar \Si=\uEnd_{\La^\op}(\Tr C)^\op$. Note that
  $\bar\Ga\cong\bar\Si$ by Lemma~\ref{le:Tr}.  In the following we use
  that $\Hom_\Si(\bar\Si,E(\bar I))\cong \bar I$ by
  Lemma~\ref{le:DTr-inj}. Also, the isomorphism \eqref{eq:DTr-inj} is
  used for $\Ga$ and $\Si$.

We fix an exact sequence
\[\xi\colon 0\lto X\xto{\ \p\ } Y\xto{\ \psi\ } Z\lto 0.\]
A projective presentation $P_1\to P_0\to C\to 0$ induces the following
commutative diagram with exact rows.
\[\begin{tikzcd}[column sep=small]
{}&&X\otimes_\La P_0^*\arrow{r}\arrow{d}{\wr}&X\otimes_\La P_1^*\arrow{r}\arrow{d}{\wr}&X\otimes_\La\Tr C\arrow{r}&0\\
0 \arrow{r}&\Hom_\La(C,X) \arrow{r}&\Hom_\La(P_0,X)
\arrow{r}&\Hom_\La(P_1,X)
\end{tikzcd}\]
Therefore $\xi$ induces the following commutative diagram with exact rows
and colums.
\[\begin{tikzcd}[column sep=small]
&0\arrow{d}&0\arrow{d}&0\arrow{d}\\
0\arrow{r}&\Hom_\La(C,X)\arrow{r}\arrow{d}&\Hom_\La(P_0,X)\arrow{r}\arrow{d}&\Hom_\La(P_1,X)\arrow{r}\arrow{d}&X\otimes_\La\Tr
C\arrow{r}\arrow{d}&0\\
0\arrow{r}&\Hom_\La(C,Y)\arrow{r}\arrow{d}&\Hom_\La(P_0,Y)\arrow{r}\arrow{d}&\Hom_\La(P_1,Y)\arrow{r}\arrow{d}&Y\otimes_\La\Tr
C\arrow{r}\arrow{d}&0\\
0\arrow{r}&\Hom_\La(C,Z)\arrow{r}&\Hom_\La(P_0,Z)\arrow{r}\arrow{d}&\Hom_\La(P_1,Z)\arrow{r}\arrow{d}&Z\otimes_\La\Tr
C\arrow{r}\arrow{d}&0\\
&&0&0&0
\end{tikzcd}\]
Now we apply the snake lemma and use adjunctions plus the isomorphism from
Lemma~\ref{le:DTr-inj}, as explained above. This yields
\begin{align*}
\Hom_\Ga(\xi^*(C),I) &=\Hom_\Ga(\Coker\Hom_\La(C,\psi),I)\\
 &\cong\Hom_\Ga(\Coker\uHom_\La(C,\psi),I)\\
&\cong\Hom_{\bar\Ga}(\Coker\uHom_\La(C,\psi),\bar I)\\
&\cong \Hom_{\bar\Si}(\Ker (\p\otimes_\La\Tr C),\bar I)\\
&\cong \Hom_{\Si}(\Ker (\p\otimes_\La\Tr C),E(\bar I))\\
&\cong \Coker \Hom_\Si(\p\otimes_\La\Tr C,E(\bar I))\\
&\cong\Coker \Hom_\La(\p,\Hom_\Si(\Tr C,E(\bar I)))\\
&=\xi_*(\t_C(I))
\end{align*}
and the proof is complete.
\end{proof}

The following result says that the dual of the transpose $\t_C(I)$ is
unique up to morphisms that factor through an injective module; it is
the representing object of the functor
\[\Hom_{\End_\La(C)}(\Ext^1_\La(C,-),I).\] For the analogues in
Auslander's work \cite{Au1978}, see Proposition~I.3.4 and
Corollary~III.4.3.

\begin{cor}\label{co:ARformula}
Let $\La$ be a ring and $C$ a finitely presented $\La$-module. Sending
an injective $\End_\La(C)$-module $I$ to $\t_C(I)$ gives a functor
\[\t_C\colon\Inj\Mod\End_\La(C)\lto\oMod\La\]
such that 
\begin{align}
\Hom_{\End_\La(C)}(\Ext^1_\La(C,-),I)&\cong\oHom_\La(-,\t_C(I))
\intertext{and}
\label{eq:Ext}\Hom_{\End_\La(C)}(\uHom_\La(C,-),I)&\cong\Ext^1_\La(-,\t_C(I)).
\end{align}
\end{cor}
\begin{proof}
  Both isomorphisms follow from the isomorphism in
  Theorem~\ref{th:defect1} by choosing an appropriate exact sequence
  $\xi\colon 0\to X\to Y\to Z\to 0$. In the first case choose $Y$ to
  be injective. Then $\xi_*\cong\oHom_\La(X,-)$ and
  $\xi^*=\Ext^1_\La(-,X)$.  In the second case choose $Y$ to be
  projective. Then $\xi^*\cong\uHom_\La(-,Z)$ and
  $\xi_*\cong\Ext^1_\La(Z,-)$. In particular, the assignment
  $I\mapsto\t_C(I)$ is functorial.
\end{proof}

We have identified the Auslander-Reiten translate $\tau_C(I)$ as the
representing object of a specific functor. The following remark
suggests that $\tau_C(I)$ is the `universal kernel' of certain
epimorphisms that are right determined by $C$.

\begin{rem}
  Fix a finitely presented $\La$-module $C$ and recall from
  \cite{Au1978} that a morphism $\a\colon X\to Y$ is \emph{right
    $C$-determined} if for any morphism $\a'\colon X'\to Y$ we have
\[\Im\Hom_\La(C,\a')\subseteq
\Im\Hom_\La(C,\a)\quad\iff\quad\text{$\a'$ factors through $\a$}.\]
Now let
\[\xi\colon 0\lto\t_C(I)\lto X\xto{\ \a\ }Y\lto 0\]
be an exact sequence of $\La$-modules and $H$ the kernel of an
$\End_\La(C)$-linear map
\[\Hom_\La(C,Y)\twoheadrightarrow \uHom_\La(C,-)\xto{\eta} I.\]
Then $\xi$ corresponds to $\eta$ under the isomorphism \eqref{eq:Ext} if
and only if $\a$ is right $C$-determined with
$H=\Im\Hom_\La(C,\a)$. This follows from the functoriality of \eqref{eq:Ext}.
\end{rem}

\begin{rem}
In Theorem~\ref{th:defect1} one can replace the finitely presented
$\La$-module $C$ by a set of finitely presented modules, as in Theorem~\ref{th:defect}.
\end{rem}

\subsection*{Examples} 
(1) Let $\La$ be a $k$-algebra over a commutative ring $k$. For a
finitely presented $\La$-module $C$ and an injective $k$-module $I$ we
have $\tau(C,I)=\Hom_k(\Tr C,I)$ and therefore
\[\Hom_k(\Ext^1_\La(C,-),I)\cong\oHom_\La(-,\t(C,I)).\]
For an elementary proof see \cite{Kr2003}.

(2) Let $\La$ be a noetherian algebra over a complete local ring $k$.
Then Matlis duality over $k$ composed with the transpose $\Tr$
identifies the noetherian $\La$-modules (modulo projectives) with the
artinian $\La$-modules (modulo injectives); see \cite{Au1978}.

(3) Let $\La$ be a Dededkind domain. Then the
Auslander-Reiten translate identifies the noetherian $\La$-modules
(modulo projectives) with the artinian $\La$-modules (modulo
injectives).

\section{Auslander-Reiten duality for quasi-coherent sheaves}

This section is devoted to giving an explicit construction of the
Auslander-Reiten translate for the category of quasi-coherent modules
over a scheme (Definition~\ref{de:translate-scheme}). In particular,
we explain the connection to Serre duality for a non-singular
projective scheme over a field.

Let $k$ be a field and $\bbX$ a quasi-compact and quasi-separated
scheme over $k$, given by a morphism $f\colon\bbX \to\bbY=\Spec k$.
Consider the category $\Qcoh\bbX$ of quasi-coherent
$\mathcal O_\bbX$-modules and note that every object is a filtered
colimit of finitely presented $\mathcal O_\bbX$-modules
\cite[I.6.9.12]{GD1971}. Thus the Grothendieck abelian category
$\Qcoh\bbX$ is locally finitely presented.

From now on assume that the scheme $\bbX$ is locally noetherian. Then
an $\mathcal O_\bbX$-module is coherent if and only if it is finitely
presented.

For an abelian category $\A$ let $\bfD(\A)$ denote its derived
category. In \cite{Ne1996}, Neeman shows that Grothendieck duality is
a formal consequence of Brown representabilty, that is, the derived
direct image functor
\[\mathbf Rf_*\colon\mathbf D(\Qcoh\bbX)\lto\mathbf D(\Qcoh\bbY)\]
has a right adjoint
\[f^!\colon\mathbf D(\Qcoh\bbY)\lto\mathbf D(\Qcoh\bbX).\]

Now fix objects $C,X\in\Qcoh\bbX$ and suppose that $C$ is
coherent. 

The following lemma may be deduced from a more general statement in
SGA 6; see Proposition~3.7 in \cite[Exp.~I]{BGI1971}. 

\begin{lem}\label{le:rigidity}
  There is in $\bfD(\Qcoh\bbX)$ a natural isomorphism
\[\s\colon X\Lotimes_{\bbX}\RHOM_\bbX(C,\mathcal O_\bbX)\xto{\sim} \RHOM_\bbX(C,X).\]
\end{lem}
\begin{proof}
Given a ring $\La$ and $\La$-modules $M,N$, there is a natural
morphism
\[N\otimes_\La\Hom_\La(M,\La)\lto\Hom_\La(M,N)\]
which is an isomorphism when $M$ is finitely generated
projective. This gives the morphism $\s$ which is an isomorphism
because $C$ is locally isomorphic to a bounded above complex of
finitely generated projective modules (since $\bbX$ is locally
noetherian) while $X$ is a bounded below complex. For the affine case,
see also \cite [Lemma~3.1]{KL2006}.
\end{proof}

\begin{lem}\label{le:gr-duality}
There is a natural isomorphism
\begin{equation*}
  \Hom_k(\RHom_{\mathbf D(\bbX)}(C,X),k)\cong
  \Hom_{\mathbf D(\bbX)}(X,\RHOM_\bbX(\RHOM_\bbX(C,\mathcal
  O_\bbX),f^!k)).
\end{equation*}
\end{lem}
\begin{proof}
  Combining Grothendieck duality, Lemma~\ref{le:rigidity}, and
  tensor-hom adjunction, we obtain the following chain of
  isomorphisms:
\begin{align*}
\Hom_k(\RHom_{\mathbf D(\bbX)}(C,X),k)
&\cong\Hom_{\mathbf D(\bbY)}(\mathbf Rf_*\RHOM_\bbX(C,X),k)\\
&\cong\Hom_{\mathbf D(\bbX)}(\RHOM_\bbX(C,X),f^!k)\\
&\cong\Hom_{\mathbf D(\bbX)}(X\Lotimes_{\bbX}\RHOM_\bbX(C,\mathcal O_\bbX),f^!k)\\
&\cong\Hom_{\mathbf D(\bbX)}(X,\RHOM_\bbX(\RHOM_\bbX(C,\mathcal O_\bbX),f^!k)).\qedhere
\end{align*}
\end{proof}

Given a complex $M$ of $\mathcal O_\bbX$-modules and $n\in\bbZ$, let
$\bfi(M)$ denote a K-injective resolution \cite{Sp1988} and set
\[Z^n(M)=\Ker (M^n\to M^{n+1}).\] 

\begin{defn}\label{de:translate-scheme}
  The \emph{Auslander-Reiten translate} of a coherent
  $\mathcal O_\bbX$-module $C$ is the $\mathcal O_\bbX$-module
\[\tau(C,k):=Z^{-1}\bfi (\RHOM_\bbX(\RHOM_\bbX(C,\mathcal
O_\bbX),f^!k)).\]
\end{defn}

The following result shows that this definition is consistent with the
notation from Corollary~\ref{co:base}. In particular, the definition
does not depend on any choice when $\tau(C,k)$ is viewed as an object
of the stable category modulo injectives.

\begin{thm}\label{th:Qcoh}
  Let $\bbX$ be a locally noetherian scheme over a
  field $k$. For a coherent $\mathcal O_\bbX$-module $C$
  there is a natural isomorphism
\begin{align*}
\oHom_{\bbX}(-,\tau(C,k)) &\xto{\sim}\Hom_k(\Ext^1_{\bbX}(C,-),k)
\intertext{and a natural monomorphism} 
\Ext^1_{\bbX}(-,\tau(C,k))&\hookrightarrow \Hom_k(\Hom_{\bbX}(C,-),k)
\end{align*}
which is an isomorphism if and only if $H^0(\RHOM_\bbX(\RHOM_\bbX(C,\mathcal
O_\bbX),f^!k))=0$.
\end{thm}
\begin{proof}
  Set $\A=\Qcoh\bbX$. The assignment $M\mapsto \bfi(M)$ yields a
  fully faithful and exact functor $\bfD(\A)\to\bfK(\Inj\A)$; see
  \cite{Sp1988}. Thus we can apply Lemma~\ref{le:cycles} and get both
  morphisms from the isomorphism in Lemma~\ref{le:gr-duality}.
\end{proof}

Let $\bbX$ be a non-singular proper scheme of dimension
$d\ge 1$ over a field $k$.  Then the above calculation simplifies
since
\[\RHOM_\bbX(\RHOM_\bbX(C,\mathcal O_\bbX),f^!k)\cong C\Lotimes_\bbX
f^!k\qquad\text{and}\qquad f^!k\cong\omega_\bbX[d]\]
where $\omega_\bbX$ denotes the dualising sheaf.  The first
isomorphism is clear since $C$ is isomorphic to a bounded complex of
locally free sheaves, and for the second isomorphism see
\cite[IV.4]{Ha1966}. Thus
\[\tau(C,k)=Z^{d-1}\bfi (C\otimes_{\bbX}\omega_\bbX).\]
Moreover, the isomorphism in Lemma~\ref{le:gr-duality} gives
\[\Hom_k(\Hom_\bbX(C,-),k)
\cong\Ext^{d}_{\bbX}(-,C\otimes_\bbX\omega_\bbX)\cong
\Ext^{1}_{\bbX}(-, \Si^{d-1}(C\otimes_\bbX\omega_\bbX))\]
where $\Si^{d-1}$ denotes the $(d-1)$st syzygy in a minimal injective
resolution.  This isomorphism is a variation of Serre duality
\cite[Exp.~XII]{Gr1968} and equals the isomorphism from
Theorem~\ref{th:second-main}. In particular, we have
$\uHom_\bbX(C,-)=\Hom_\bbX(C,-)$.

The following is an application of Theorem~\ref{th:second-main}.

\begin{prop}\label{pr:proj-scheme}
  Let $\bbX$ be a non-singular proper scheme of dimension $d\ge 1$
  over a field, and fix a pair of $\mathcal O_\bbX$-modules $C,X$ such
  that $C$ is coherent. Then there exists an epimorphism
  $\pi\colon X_C\to X$ such that for every morphism $\p\colon C\to X$
  the following holds:
\[\p \text{ factors through }\pi\quad\iff
\quad\Ext^1_\bbX(\p,-)=0\quad\iff\quad\p=0.\]
\end{prop}
\begin{proof}
The construction of $\pi$ is given in the proof of Theorem~\ref{th:second-main}.
\end{proof}

\begin{rem}
  (1) There is a canonical choice for $\pi\colon X_C\to X$;
  see Remark~\ref{re:pi}.
  
  (2) When $X$ is coherent, then there is a choice for
  $\pi\colon X_C\to X$ such that $X_C$ is coherent.

  (3) One may conjecture that the second morphism in
  Theorem~\ref{th:Qcoh} induces an isomorphism
\[\Ext^1_{\bbX}(-,\tau(C,k))\xto{\sim} \Hom_k(\uHom_{\bbX}(C,-),k).\]
\end{rem}

\subsection*{Some questions}

Let $\A$ be a Grothendieck abelian category.

(1) Is there an alternative construction of the Auslander-Reiten
translate $\t_C$ for a finitely presented object $C\in\A$ that is not based
on the existence of flat covers?

(2) Are there examples when the morphism
\[\Ext^1_\A(-,\t_C(I))\lto \Hom_\Ga(\uHom_\A(C,-),I)\] from
Theorem~\ref{th:second-main} is not invertible?

(3) Suppose that $\fp\A$ is $k$-linear and Ext-finite over a field $k$. When is the
Auslander-Reiten translate $\tau(C,k)$ finitely presented for all $C$
in $\fp\A$? And when does the Auslander-Reiten translate induce an
equivalence between the projectively stable and the injectively stable
category associated with $\fp\A$?

An important class where both properties hold are given by categories
$\A$ such that $\fp\A$ is a dualising $k$-variety
\cite{AR1974}. However, weighted projective lines \cite{GL1987} provide interesting
examples where these properties hold but $\fp\A$ is not a dualising
$k$-variety; see also \cite{CL2015,LZ2004}.

(4) Suppose that $\A$ is locally noetherian and that $\fp\A$ is
$k$-linear and Ext-finite over a complete local ring $k$. When does
the Auslander-Reiten translate (via Matlis duality over $k$) induce an
equivalence between the projectively stable category of noetherian
objects and the injectively stable category of artinian objects?

\begin{appendix}
\section{Complexes of injectives and the stable category}

Let $\A$ be an abelian category and let $\Inj\A$ denote the full
subcategory of injective objects.  The chain complexes in $\Inj\A$
with morphisms the chain maps up to homotopy are denoted by
$\bfK(\Inj\A)$.

Suppose that $\A$ has enough injective objects. Then we denote by
\[\bfi\colon \A\lto\bfK(\Inj\A)\]
the fully faithful functor that takes an object in $\A$ to an injective resolution.

For an integer $n$ consider the functor
\[Z^n\colon\bfK(\Inj\A)\lto\overline\A,\qquad X\mapsto\Ker (X^n\to
X^{n+1}).\]
Note that  for  $X\in\bfK(\Inj\A)$ there is a natural morphism
$\bfi Z^0(X)\to X$.

\begin{lem}\label{le:cycles}
 For objects $X\in\A$ and $Y\in\bfK(\Inj\A)$ the following holds.
\begin{enumerate}
\item If $\Hom_{\bfK(\Inj\A)}(-,Y)$ vanishes on complexes concentrated
  in degree zero, then $Z^0$ induces a natural isomorphism
\[\Hom_{\bfK(\Inj\A)}(\mathbf i(X),Y)\xto{\sim}
\oHom_{\A}(X,Z^{0}(Y)).\]
\item There is a natural monomorphism
\[\Ext^1_\A(X,Z^{-1}(Y))\hookrightarrow \Hom_{\bfK(\Inj\A)}(\mathbf i(X),Y)\]
which is an isomorphism for all $X\in\A$ if and only if $H^0(Y)=0$.
\end{enumerate}
\end{lem}
\begin{proof}
  The proof is straightforward. The second morphism is the composition
  of
\[\Ext^1_\A(X,Z^{-1}(Y))\xto{\sim}\Hom_{\bfK(\Inj\A)}(\bfi(X),\bfi Z^{-1}(Y)[1])\]
and the morphism induced by $\bfi Z^{-1}(Y)[1]\to Y$.
\end{proof}

\section{Natural maps of extension functors}\label{ap:ext}

We rewrite the isomorphisms \eqref{eq:repres} and
\eqref{eq:proj-repres} from the introduction in terms of natural
transformations between extension functors. This is based on
Lemma~\ref{le:hom-ext} and reveals the symmetry between both formulas.

Let $\A$ be a Grothendieck abelian category that is locally finitely
presented. Fix a finitely presented object $C$ and and injective
module $I$ over $\Ga=\End_\A(C)$. Then for $X\in\A$ there is a natural
isomorphism
\begin{align*}
\Hom_\Ga(\Hom(\Hom_\A(X,-),\Ext^1_\A(C,-)),I)\cong\Hom(\Ext^1_\A(-,X),\Ext^1_\A(-,\t_C(I))).
\intertext{When $\A$ has enough projective morphisms we have also}
\Hom_\Ga(\Hom(\Ext^1_\A(X,-),\Ext^1_\A(C,-)),I)\cong\Hom(\Hom_\A(-,X),\Ext^1_\A(-,\t_C(I))).
\end{align*}

\end{appendix}

\subsection*{Acknowledgements}
I am grateful to Amnon Neeman for his help with the algebraic geometry
in this paper.

\end{document}